\documentclass[12pt,reqno]{amsart}

\usepackage{amsfonts, amsthm, amsmath}
\allowdisplaybreaks[4]

\usepackage{rotating}

\usepackage{tikz}

\usepackage{graphics}

\usepackage{amssymb}

\usepackage{amscd}

\usepackage[latin2]{inputenc}

\usepackage{t1enc}

\usepackage[mathscr]{eucal}

\usepackage{indentfirst}

\usepackage{graphicx}

\usepackage{graphics}

\usepackage{pict2e}

\usepackage{mathrsfs}

\usepackage{enumerate}
\usepackage[pagebackref]{hyperref}
\hypersetup{backref, pagebackref, colorlinks=true}
\usepackage{cite}
\usepackage{color}
\usepackage{epic}
\usepackage{hyperref} %this gives clickable references, which is nice
\numberwithin{equation}{section}
\theoremstyle{plain}

\newtheorem{theorem}{Theorem}[section]

\newtheorem{lemma}[theorem]{Lemma}

\newtheorem{corollary}[theorem]{Corollary}

\newtheorem{conjecture}[theorem]{Conjecture}

\theoremstyle{definition}

\newtheorem{remark}[theorem]{Remark}

\newtheorem{?}[theorem]{Problem}

\def\boxit#1{\leavevmode\hbox{\vrule\vtop{\vbox{\kern.33333pt\hrule
    \kern1pt\hbox{\kern1pt\vbox{#1}\kern1pt}}\kern1pt\hrule}\vrule}}

\usepackage{collectbox}



\newcommand{\f}[1]{\ifthenelse{\equal{#1}{1}}{(q;q)_\infty}{(q^{#1};q^{#1})_{\infty}}}

\begin{document}
\title[Several $q$-series related to Ramanujan's theta functions]{Several $q$-series related to Ramanujan's theta functions}

\author[D. Tang]{Dazhao Tang}

\address[Dazhao Tang]{College of Mathematics and Statistics, Chongqing University, Huxi Campus LD204, Chongqing 401331, P.R. China}
\email{dazhaotang@sina.com}

\author[E. X. W. Xia]{Ernest. X. W. Xia}
\address[Ernest X. W. Xia]{Department of Mathematics, Jiangsu University, Zhenjiang, Jiangsu, 212013, P.R. China}
\email{Ernestxwxia@163.com}

\date{\today}

\begin{abstract}
Quite recently, the first author investigated vanishing coefficients of the arithmetic progressions in several $q$-series expansions. In this paper, we further study the signs of coefficients in two $q$-series expansions and establish some arithmetic relations for several $q$-series expansions by means of Ramanujan's theta functions. We obtain the 5-dissections of these two $q$-series and give combinatorial interpretations for these dissections. Moreover, we obtain four $q$-series identities involving the aforementioned $q$-series, two of which were proved by Kim and Toh via modular forms.
\end{abstract}

\subjclass[2010]{33D15, 11F33, 30B10}

\keywords{Ramanujan's theta functions; $q$-series expansions; Jacobi's triple product identity; arithmetic relations}

\maketitle

%%%%%%%%%%%%%%%%%%%
\section{Introduction}
Quite recently, Hirschhorn \cite{Hir2018} investigated vanishing coefficients of the arithmetic progressions in two $q$-series expansions. Motivated by the work of Hirschhorn, the first author \cite{Tang2018} investigated vanishing coefficients of the arithmetic progressions in following $q$-series expansions
\begin{align}
(-q,-q^{4};q^{5})_{\infty}^{2}(q^{4},q^{6};q^{10})_{\infty} &=\sum_{n=0}^{\infty}g_{1}(n)q^{n},\label{gf:g1}\\
(-q^{2},-q^{3};q^{5})_{\infty}^{2}(q^{2},q^{8};q^{10})_{\infty} &=\sum_{n=0}^{\infty}h_{1}(n)q^{n}.\label{gf:h1}
\end{align}
Here and in the sequel, we adopt the following standard $q$-series notation:
\begin{align*}
(a;q)_{\infty} &:=\prod_{n=0}^{\infty}(1-aq^{n}),\\
(a_{1},a_{2},\ldots,a_{m};q)_{\infty} &:=(a_{1};q)_{\infty}(a_{2};q)_{\infty}\cdots(a_{m};q)_{\infty},\quad\textrm{for}~~|q|<1.
\end{align*}

In \cite[Eqs. (1.3) and (1.4)]{Tang2018}, the first author proved that for $n\geq0$,
\begin{align}
g_{1}(5n+3) &=h_{1}(5n+1)=0.\label{iden-vanish}
\end{align}

Moreover, the first author conjectured the signs of coefficients in $q$-series \eqref{gf:h1} are periodic from some $n$. In this paper, we not only confirm this conjecture, but also establish 5-dissections of \eqref{gf:g1} and \eqref{gf:h1} along with combinatorial interpretations for these dissections.

Firstly, we obtain the following 5-dissections of \eqref{gf:g1} and \eqref{gf:h1}.
\begin{theorem}\label{beau THM-1}
We have
\begin{align*}
(-q,-q^{4};q^{5})_{\infty}^{2}(q^{4},q^{6};q^{10})_{\infty}=G_{0}(q^{5})+qG_{1}(q^{5})+q^{2}G_{2}(q^{5})+q^{4}G_{4}(q^{5}),
\end{align*}
where
\begin{align}
G_{0}(q) &=\sum_{n=0}^{\infty}g_{1}(5n)q^{n}=\dfrac{1}{(q,q^{4};q^{5})_{\infty}^{2}(q^{2},q^{8};q^{10})_{\infty}},\label{a1:5n}\\
G_{1}(q) &=\sum_{n=0}^{\infty}g_{1}(5n+1)q^{n}=\dfrac{2}{(q,q^{2},q^{3},q^{4};q^{5})_{\infty}(q^{2},q^{8};q^{10})_{\infty}},\label{a1:5n+1}\\
G_{2}(q) &=\sum_{n=0}^{\infty}g_{1}(5n+2)q^{n}=\dfrac{1}{(q,q^{4};q^{5})_{\infty}^{2}(q^{4},q^{6};q^{10})_{\infty}},\label{a1:5n+2}\\
G_{4}(q) &=\sum_{n=0}^{\infty}g_{1}(5n+4)q^{n}=\dfrac{1}{(q^{2},q^{3};q^{5})_{\infty}^{2}(q^{4},q^{6};q^{10})_{\infty}}.\label{a1:5n+4}
\end{align}
\end{theorem}

\begin{theorem}\label{beau THM-2}
We have
\begin{align*}
(-q^{2},-q^{3};q^{5})_{\infty}^{2}(q^{2},q^{8};q^{10})_{\infty}=H_{0}(q^{5})+q^{2}H_{2}(q^{5})+q^{3}H_{3}(q^{5})+q^{4}H_{4}(q^{5}),
\end{align*}
where
\begin{align}
H_{0}(q) &=\sum_{n=0}^{\infty}h_{1}(5n)q^{n}=\dfrac{1}{(q,q^{4};q^{5})_{\infty}^{2}(q^{2},q^{8};q^{10})_{\infty}},\label{b1:5n}\\
H_{2}(q) &=\sum_{n=0}^{\infty}h_{1}(5n+2)q^{n}=\dfrac{1}{(q^{2},q^{3};q^{5})_{\infty}^{2}(q^{2},q^{8};q^{10})_{\infty}},\label{b1:5n+1}\\
H_{3}(q) &=\sum_{n=0}^{\infty}h_{1}(5n+3)q^{n}=\dfrac{2}{(q,q^{2},q^{3},q^{4};q^{5})_{\infty}(q^{4},q^{6};q^{10})_{\infty}},\label{b1:5n+2}\\
H_{4}(q) &=\sum_{n=0}^{\infty}h_{1}(5n+4)q^{n}=\dfrac{-1}{(q^{2},q^{3};q^{5})_{\infty}^{2}(q^{4},q^{6};q^{10})_{\infty}}.\label{b1:5n+4}
\end{align}
\end{theorem}

Therefore we get the following combinatorial interpretations.

$g_{1}(5n)$ is the number of partitions of $n$ into parts which are $\pm1$, $\pm2$, $\pm4\pmod{10}$, where parts $\pm1$ and $\pm4$ appear in two flavours,

$g_{1}(5n+1)$ is the twice of number of partitions of $n$ into parts which are $\pm1$, $\pm2$, $\pm3$, $\pm4\pmod{10}$, where parts $\pm2$ appear in two flavours,

$g_{1}(5n+2)$ is the number of partitions of $n$ into parts which are $\pm1$, $\pm4\pmod{10}$, where parts $\pm1$ parts appear in two flavours and $\pm4$ appear in three flavours,

$g_{1}(5n+4)$ is the number of partitions of $n$ into parts which are $\pm2$, $\pm3$, $\pm4\pmod{10}$, where parts $\pm2$ and $\pm3$ appear in two flavours,

$h_{1}(5n)$ is the number of partitions of $n$ into parts which are $\pm1$, $\pm2$, $\pm4\pmod{10}$, where parts $\pm1$ and $\pm4$ appear in two flavours,

$h_{1}(5n+2)$ is the number of partitions of $n$ into parts which are $\pm2$, $\pm3\pmod{10}$, where parts $\pm3$ parts appear in two flavours and $\pm2$ appear in three flavours,

$h_{1}(5n+3)$ is the twice of number of partitions of $n$ into parts which are $\pm1$, $\pm2$, $\pm3$, $\pm4\pmod{10}$, where parts $\pm4$ appear in two flavours,

$-h_{1}(5n+4)$ is the number of partitions of $n$ into parts which are $\pm2$, $\pm3$, $\pm4\pmod{10}$, where parts $\pm2$ and $\pm3$ appear in two flavours.

By these combinatorial interpretations, we obtain immediately the following corollaries.
\begin{corollary}
For any integer $n\geq0$,
\begin{align*}
g_{1}(5n) &>0,\\
g_{1}(5n+1) &>0,\\
g_{1}(5n+2) &>0,\\
g_{1}(5n+4) &>0 \quad(n\neq1).
\end{align*}
\end{corollary}

\begin{corollary}
For any integer $n\geq0$,
\begin{align*}
h_{1}(5n) &>0,\\
h_{1}(5n+2) &>0 \quad(n\neq1),\\
h_{1}(5n+3) &>0,\\
h_{1}(5n+4) &<0 \quad(n\neq1).
\end{align*}
\end{corollary}

\begin{corollary}
For any integer $n\geq0$,
\begin{align}
g_{1}(5n) &=h_{1}(5n),\label{a1b1-1}\\
g_{1}(5n+4) &=-h_{1}(5n+4).\label{a1b1-2}
\end{align}
\end{corollary}

Moreover, the first author studied vanishing coefficients in following two general $q$-series expansions:
\begin{align}
(-q^{r},-q^{t-r};q^{t})_{\infty}^{3}(q^{s},q^{2t-s};q^{2t})_{\infty} &:=\sum_{n=0}^{\infty}g_{r,s,t}(n)q^{n},\label{g,r-s-t}\\
(-q^{r},-q^{t-r};q^{t})_{\infty}(q^{s},q^{2t-s};q^{2t})_{\infty}^{3} &:=\sum_{n=0}^{\infty}h_{r,s,t}(n)q^{n}\label{h,r-s-t}
\end{align}
where $t\geq5$ is a prime, $r ,s$ are positive integers and $r<t$, $s\neq t$.

Interestingly, we obtain the following identities of $q$-series expansions \eqref{g,r-s-t} and \eqref{h,r-s-t} for $t=5$, which parallel to \eqref{a1b1-1} and \eqref{a1b1-2}.
\begin{theorem}\label{THM:relation}
For any integer $n\geq0$,
\begin{align}
g_{1,2,5}(5n+1) &=g_{2,4,5}(5n+2),\label{a2b2-relation1}\\
g_{1,2,5}(5n+3) &=-g_{2,4,5}(5n+4),\label{a2b2-relation2}\\
g_{1,3,5}(5n) &=g_{2,1,5}(5n),\label{a3b3-relation1}\\
g_{1,3,5}(5n+2) &=g_{2,1,5}(5n+2).\label{a3b3-relation2}\\
h_{1,1,5}(5n) &=h_{2,3,5}(5n+2),\label{analog-relat-1}\\
h_{1,1,5}(5n+1) &=h_{2,3,5}(5n+3),\label{analog-relati-2}\\
h_{1,4,5}(5n+1) &=h_{2,2,5}(5n),\\
h_{1,4,5}(5n+2) &=-h_{2,2,5}(5n+1).\label{final-relation1}
\end{align}
\end{theorem}

Finally, we define the following two $q$-series expansion
\begin{align}
(q,q^{4};q^{5})_{\infty}^{2}(q^{4},q^{6};q^{10})_{\infty} &=\sum_{n=0}^{\infty}g_{2}(n)q^{n},\label{gf:g2}\\
(q^{2},q^{3};q^{5})_{\infty}^{2}(q^{2},q^{8};q^{10})_{\infty} &=\sum_{n=0}^{\infty}h_{2}(n)q^{n}.\label{gf:h2}
\end{align}

We also obtain several $q$-series identities involving \eqref{gf:g1}, \eqref{gf:h1}, \eqref{h,r-s-t}, \eqref{gf:g2}, and \eqref{gf:h2}.
\begin{theorem}\label{THM:relation-2}
We have
\begin{align}
(-q,-q^{4};q^{5})_{\infty}^{2}(q^{4},q^{6};q^{10})_{\infty} &+(-q^{2},-q^{3};q^{5})_{\infty}^{2}(q^{2},q^{8};q^{10})_{\infty}\notag\\
 &\quad=\dfrac{2(q^{10};q^{10})_{\infty}^{3}}{(q^{2};q^{2})_{\infty}(q^{5};q^{5})_{\infty}^{2}}(-q,-q^{4};q^{5})_{\infty}(q^{4},q^{6};q^{10})_{\infty}^{3},
 \label{q-iden-1}\\
(q,q^{4};q^{5})_{\infty}^{2}(q^{4},q^{6};q^{10})_{\infty} &+(q^{2},q^{3};q^{5})_{\infty}^{2}(q^{2},q^{8};q^{10})_{\infty}\notag\\
 &\quad=\dfrac{2(q;q)_{\infty}^{2}(q^{10};q^{10})_{\infty}^{4}}{(q^{2};q^{2})_{\infty}^{2}
 (q^{5};q^{5})_{\infty}^{4}}(-q,-q^{4};q^{5})_{\infty}(q^{4},q^{6};q^{10})_{\infty}^{3},\label{q-iden-2}\\
(-q,-q^{4};q^{5})_{\infty}(q^{4},q^{6};q^{10})_{\infty}^{3} &-q(-q^{2},-q^{3};q^{5})_{\infty}(q^{2},q^{8};q^{10})_{\infty}^{3}\notag\\
 &\quad=\dfrac{(q^{2};q^{2})_{\infty}(q^{5};q^{5})_{\infty}^{2}}{(q^{10};q^{10})_{\infty}^{3}}(-q^{2},-q^{3};q^{5})_{\infty}^{2}(q^{2},q^{8};q^{10})_{\infty},
 \label{q-iden-3}\\
(-q,-q^{4};q^{5})_{\infty}(q^{4},q^{6};q^{10})_{\infty}^{3} &+q(-q^{2},-q^{3};q^{5})_{\infty}(q^{2},q^{8};q^{10})_{\infty}^{3}\notag\\
 &\quad=\dfrac{(q^{2};q^{2})_{\infty}^{2}(q^{5};q^{5})_{\infty}^{4}}
 {(q;q)_{\infty}^{2}(q^{10};q^{10})_{\infty}^{4}}(q^{2},q^{3};q^{5})_{\infty}^{2}(q^{2},q^{8};q^{10})_{\infty}.\label{q-iden-4}
\end{align}
\end{theorem}

\begin{remark}
Very recently, Kim and Toh \cite[Lemma 3.1]{KT2018} proved the following two $q$-series identities via modular forms:
\begin{align}
\dfrac{(-q^{2},-q^{3},q^{5};q^{5})_{\infty}^{2}(q^{10};q^{10})_{\infty}}{(q^{4},q^{6};q^{10})_{\infty}} &+\dfrac{(-q,-q^{4},q^{5};q^{5})_{\infty}^{2}(q^{10};q^{10})_{\infty}}{(q^{2},q^{8};q^{10})_{\infty}}\nonumber\\
 &\quad=\dfrac{2(-q,-q^{4},q^{5};q^{5})_{\infty}(q^{2};q^{2})_{\infty}(q^{10};q^{10})_{\infty}^{2}}{(q^{2},q^{8};q^{10})_{\infty}^{3}(q^{5};q^{5})_{\infty}},
 \label{KT-iden-1}\\
\dfrac{(-q^{2},-q^{3},q^{5};q^{5})_{\infty}^{2}(q^{10};q^{10})_{\infty}}{(q^{4},q^{6};q^{10})_{\infty}} &+q\dfrac{(-q^{2},-q^{3},q^{5};q^{5})_{\infty}(q^{2};q^{2})_{\infty}(q^{10};q^{10})_{\infty}^{2}}{(q^{4},q^{6};q^{10})_{\infty}^{3}(q^{5};q^{5})_{\infty}}\nonumber\\
 &\quad=\dfrac{(-q,-q^{4},q^{5};q^{5})_{\infty}(q^{2};q^{2})_{\infty}(q^{10};q^{10})_{\infty}}{(q^{2},q^{8};q^{10})_{\infty}^{3}(q^{5};q^{5})_{\infty}}.\label{KT-iden-2}
\end{align}
Interestingly, \eqref{q-iden-1} and \eqref{q-iden-3} are equivalent to \eqref{KT-iden-1} and \eqref{KT-iden-2}, respectively.
\end{remark}

The rest of this paper is constructed as follows. In Sect. \ref{sec:lemmas}, we introduce some necessary notation as well as identities involving theta functions $\varphi(q)$ and $\psi(q)$. In Sect. \ref{sec:two theorems}, we prove Theorems \ref{beau THM-1} and \ref{beau THM-2}. The proofs of Theorems \ref{THM:relation} and \ref{THM:relation-2} are given in Sect. \ref{sec:THM-relation}. We conclude in the last section with some remarks to motivate further investigation.

\section{Preliminary results}\label{sec:lemmas}
Ramanujan's general theta function is defined by
\begin{align*}
f(a,b) &:=\sum_{n=-\infty}^{\infty}a^{n(n+1)/2}b^{n(n-1)/2},\quad |ab|<1.
\end{align*}
Basic properties enjoyed by $f(a,b)$ proved in \cite[p. 34, Entry 18]{Ber1991}
\begin{align}
f(a,b) &=f(b,a),\nonumber\\
f(1,a) &=2f(a,a^{3}).\label{iden:psi}
\end{align}
The function $f(a,b)$ satisfies the well-known Jacobi triple product identity \cite[p. 35, Entry 19]{Ber1991}:
\begin{align}
f(a,b) &=(-a,-b,ab;ab)_{\infty}.\label{JTP-identity}
\end{align}
Eq. \eqref{JTP-identity} is used frequently and without mention in the sequel.

The two important special cases of \eqref{JTP-identity} are \cite[Eqs. (1.5.4) and (1.5.5)]{Hirb2017}
\begin{align}
\varphi(q) &:=f(q,q)=\sum_{n=-\infty}^{\infty}q^{n^{2}}=\dfrac{(q^{2};q^{2})_{\infty}^{5}}{(q;q)_{\infty}^{2}(q^{4};q^{4})_{\infty}^{2}},\nonumber\\
\psi(q) &:=f(q,q^{3})=\sum_{n=0}^{\infty}q^{n(n+1)/2}=\dfrac{(q^{2};q^{2})_{\infty}^{2}}{(q;q)_{\infty}}.\label{def:psi}
\end{align}

\begin{lemma}
We have
\begin{align}
\varphi(q) &=\varphi(q^4)+2q\psi(q^8),\label{2-dissec-phi}\\
4q(q^{4};q^{4})_{\infty}(q^{20};q^{20})_{\infty} &=\varphi(q)\varphi(-q^5)-\varphi(-q)\varphi(q^5).\label{relati:phi-psi}
\end{align}
\end{lemma}
\begin{proof}
Eq. \eqref{2-dissec-phi} follows from \cite[p. 40, Entry 25 (i), (ii)]{Ber1991} and Eq. \eqref{relati:phi-psi} appears in \cite[p. 278]{Ber1991}.
\end{proof}

The following lemma is the main ingredient for our proof.
\begin{lemma}
If $ab=cd$, then
\begin{align}
f(a,b)f(c,d)=f(ac,bd)f(ad,bc)+af\left(\frac{b}{c},\frac{c}{b}abcd\right)f\left(\frac{b}{d},\frac{d}{b}abcd\right).\label{ff}
\end{align}
\end{lemma}
\begin{proof}
Eq. \eqref{ff} comes from \cite[p. 45, Entry 29]{Ber1991} and \cite[p. 9, Theorem 0.6]{Coob2017}.
\end{proof}

Finally, we need the following two identities involving $\varphi(q)$ and $\psi(q)$.
\begin{lemma}
We have
\begin{align}
\varphi(q)-\varphi(q^{5}) &=2q\dfrac{(q^{4},q^{6},q^{10},q^{14},q^{16},q^{20};q^{20})_{\infty}}{(q^{3},q^{7},q^{8},q^{12},q^{13},q^{17};q^{20})_{\infty}},\label{phi-psi-1}\\
\psi(q^{2})-q\psi(q^{10}) &=\dfrac{(q,q^{9},q^{10},q^{11},q^{19},q^{20};q^{20})_{\infty}}{(q^{2},q^{3},q^{7},q^{13},q^{17},q^{18};q^{20})_{\infty}}.\label{phi-psi-2}
\end{align}
\end{lemma}
\begin{proof}
Eqs. \eqref{phi-psi-1} and \eqref{phi-psi-2} are proved in \cite[p. 311, Eqs. (34.1.8) and (34.1.12)]{Hirb2017}.
\end{proof}

\section{Proofs of Theorems \ref{beau THM-1} and \ref{beau THM-2}}\label{sec:two theorems}
To obtain \eqref{a1:5n}--\eqref{a1:5n+4}, we first prove two necessary lemmas.

Let $k>0, l\geq0$ be integers and let $G(q)=\sum\limits_{n=0}^{\infty}g(n)q^{n}$ be a formal power series. Define an operator $H_{k,l}$ by
\begin{align*}
H_{k,l}\left(G(q)\right) &:=\sum_{n=0}^{\infty}g(kn+l)q^{kn+l}.
\end{align*}

\begin{lemma}
Define
\begin{align*}
M_{1}(q) &:=f(q^{18},q^{22})^{2}-q^{8}f(q^{2},q^{38})^{2},\\
N_{1}(q) &:=q^{5}f(q^{12},q^{28})f(q^{2},q^{48})+q^{6}f(q^{8},q^{32})f(q^{2},q^{48})\\
 &\quad-qf(q^{12},q^{28})f(q^{18},q^{22})-q^{2}f(q^{8},q^{32})f(q^{18},q^{22}).
\end{align*}
Then
\begin{align}
\varphi(q)M_{1}(q)+2\psi(q^{2})N_{1}(q) =\varphi(q^{5})M_{1}(q)+2q\psi(q^{10})N_{1}(q).\label{iden:M-N}
\end{align}
\end{lemma}

\begin{proof}
Putting $(a,b,c,d)=(-q^{8},-q^{12},-q^{10},-q^{10})$ in \eqref{ff}, we get
\begin{align}
M_{1}(q)=f(q^{18},q^{22})^{2}-q^{8}f(q^{2},q^{38})^{2}=f(-q^{8},-q^{12})f(-q^{10},-q^{10}).\label{represe:M}
\end{align}
Similarly, taking $(a,b,c,d)=(-q^{5},-q^{15},-q^{7},-q^{13})$ in \eqref{ff},
\begin{align}
qf(q^{12},q^{28})f(q^{18},q^{22})-q^{6}f(q^{8},q^{32})f(q^{2},q^{48})=qf(-q^{5},-q^{15})f(-q^{7},-q^{13}).\label{iden:N1}
\end{align}
Picking $(a,b,c,d)=(-q^{3},-q^{17},-q^{5},-q^{15})$ in \eqref{ff},
\begin{align}
q^{2}f(q^{8},q^{32})f(q^{18},q^{22})-q^{5}f(q^{8},q^{32})f(q^{2},q^{48})=q^{2}f(-q^{3},-q^{17})f(-q^{5},-q^{15}).\label{iden:N2}
\end{align}
Finally, taking $(a,b,c,d)=(q,q^{9},-q^{4},-q^{6})$ in \eqref{ff},
\begin{align}
f(-q^{5},-q^{15})f(-q^{7},-q^{13})+qf(-q^{3},-q^{17})f(-q^{5},-q^{15})=f(q,q^{9})f(-q^{4},-q^{6}).\label{iden:N3}
\end{align}
Employing \eqref{iden:N1}--\eqref{iden:N3}, we readily obtain
\begin{align}
N_{1}(q)=-qf(q,q^{9})f(-q^{4},-q^{6}).\label{represe:N}
\end{align}
Now, with the help of \eqref{phi-psi-1}, \eqref{phi-psi-2}, \eqref{represe:M}, and \eqref{represe:N},
\begin{align*}
 &\left(\varphi(q)-\varphi(q^{5})\right)M_{1}(q)+2\left(\psi(q^{2})-q\psi(q^{10})\right)N_{1}(q)\\
 =&2q\dfrac{(q^{4},q^{6},q^{10},q^{14},q^{16},q^{20};q^{20})_{\infty}}{(q^{3},q^{7},q^{8},q^{12},q^{13},q^{17};q^{20})_{\infty}}
 \times(q^{8},q^{10},q^{10},q^{12},q^{20},q^{20};q^{20})_{\infty}\\
 &\quad-2q\dfrac{(q,q^{9},q^{10},q^{11},q^{19},q^{20};q^{20})_{\infty}}{(q^{2},q^{3},q^{7},q^{13},q^{17},q^{18};q^{20})_{\infty}}\times
 (-q,q^{4},q^{6},-q^{9},q^{10},q^{10};q^{10})_{\infty}\\
 =&\dfrac{(q^{4},q^{6},q^{10},q^{10},q^{10};q^{10})_{\infty}}{(q^{3},q^{7};q^{10})_{\infty}}-\dfrac{(q^{4},q^{6},q^{10},q^{10},q^{10};q^{10})_{\infty}}
{(q^{3},q^{7};q^{10})_{\infty}}\\
 =&0,
\end{align*}
as desired.
\end{proof}

\begin{lemma}
We have
\begin{align}
\psi(q^{2})\varphi(q^{5})-q\varphi(q)\psi(q^{10})=(q;q)_{\infty}(q^{5};q^{5})_{\infty}.\label{phi-psi}
\end{align}
\end{lemma}
\begin{proof}
Firstly, replacing $q$ by $-q^{5}$ in \eqref{2-dissec-phi}, we find that
\begin{align}
\varphi(-q^5)=\varphi(q^{20})-2q^5\psi(q^{40}).\label{5-identity}
\end{align}
Combining \eqref{2-dissec-phi} and \eqref{5-identity} yields
\begin{align}
\varphi(q)\varphi(-q^5)=\varphi(q^4)\varphi(q^{20})+2q\psi(q^8)\varphi(q^{20})-2q^5\varphi(q^4)\psi(q^{40})-4q^6\psi(q^8)\psi(q^{40}).\label{relati:phi-psi-2}
\end{align}
Replacing $q$ by $-q$ in \eqref{relati:phi-psi-2},
\begin{align}
\varphi(-q)\varphi(q^5)=\varphi(q^4)\varphi(q^{20})-2q\psi(q^8)\varphi(q^{20})+2q^5\varphi(q^4)\psi(q^{40})-4q^6\psi(q^8)\psi(q^{40}).\label{relati:phi-psi-3}
\end{align}
By \eqref{relati:phi-psi-2} and \eqref{relati:phi-psi-3},
\begin{align}
\varphi(q)\varphi(-q^5)-\varphi(-q)\varphi(q^5)=4q\psi(q^8)\varphi(q^{20})-4q^5\varphi(q^4)\psi(q^{40}).\label{relati:phi-psi-4}
\end{align}
Finally, substituting \eqref{relati:phi-psi} into \eqref{relati:phi-psi-4} and replacing $q$ by $q^{1/4}$, we obtain \eqref{phi-psi}.
\end{proof}

Now we turn to prove \eqref{a1:5n}.

On one hand, according to \cite{Tang2018}, we find that
\begin{align}
\sum_{n=0}^{\infty}g_{1}(n)q^{n} &=\dfrac{\varphi(q^{5})}{(q^{5};q^{5})_{\infty}^{2}(q^{10};q^{10})_{\infty}}\big(S_{1}-q^{4}S_{2}+q^{2}S_{3}-q^{6}S_{4}\big)\nonumber\\
 &\quad+\dfrac{2\psi(q^{10})}{(q^{5};q^{5})_{\infty}^{2}(q^{10};q^{10})_{\infty}}\big(qS_{5}-q^{5}S_{6}+q^{4}S_{7}-q^{8}S_{8}\big),\label{gf:all a}
\end{align}
where
\begin{align*}
S_{1} &=\sum_{m,n=-\infty}^{\infty}q^{20m^{2}+2m+20n^{2}+6n},\quad S_{2}=\sum_{m,n=-\infty}^{\infty}q^{20m^{2}+18m+20n^{2}+6n},\\
S_{3} &=\sum_{m,n=-\infty}^{\infty}q^{20m^{2}+2m+20n^{2}+14n},\quad S_{4}=\sum_{m,n=-\infty}^{\infty}q^{20m^{2}+18m+20n^{2}+14n},\\
S_{5} &=\sum_{m,n=-\infty}^{\infty}q^{20m^{2}+2m+20n^{2}+4n},\quad S_{6}=\sum_{m,n=-\infty}^{\infty}q^{20m^{2}+18m+20n^{2}+4n},\\
S_{7} &=\sum_{m,n=-\infty}^{\infty}q^{20m^{2}+2m+20n^{2}+16n},\quad S_{8}=\sum_{m,n=-\infty}^{\infty}q^{20m^{2}+18m+20n^{2}+16n}.
\end{align*}

In $S_{1}$, if $2m+6n\equiv0\pmod{5}$, then $2m+n\equiv0\pmod{5}$. Equivalently, $m-2n\equiv0\pmod{5}$. Assume $2m+n=5r$ and $m-2n=-5s$, it follows that $m=2r-s$ and $n=r+2s$. Therefore
\begin{align}
H_{5,0}(S_{1})=\sum_{r,s=-\infty}^{\infty}q^{100r^{2}+10r+100s^{2}+10s}=f(q^{90},q^{110})^{2}.\label{S1:5}
\end{align}

Similarly, we obtain
\begin{align}
H_{5,0}(q^{4}S_{2}) &=q^{20}f(q^{10},q^{190})f(q^{90},q^{110}),\label{S2:5}\\
H_{5,0}(q^{2}S_{3}) &=q^{20}f(q^{10},q^{190})f(q^{90},q^{110}),\label{S3:5}\\
H_{5,0}(q^{6}S_{4}) &=q^{40}f(q^{10},q^{190})^{2},\label{S4:5}\\
H_{5,0}(qS_{5}) &=q^{25}f(q^{60},q^{140})f(q^{10},q^{190}),\label{S5:5}\\
H_{5,0}(q^{5}S_{6}) &=q^{5}f(q^{60},q^{140})f(q^{90},q^{110}),\label{S6:5}\\
H_{5,0}(q^{4}S_{7}) &=q^{30}f(q^{40},q^{160})f(q^{10},q^{190}),\label{S7:5}\\
H_{5,0}(q^{8}S_{8}) &=q^{10}f(q^{40},q^{160})f(q^{90},q^{110}).\label{S8:5}
\end{align}
Picking out the term involving $q^{5n}$ in \eqref{gf:all a}, applying \eqref{S1:5}--\eqref{S8:5} and replacing $q^{5}$ by $q$, we obtain
\begin{align}\label{gf:a5n}
\sum_{n=0}^{\infty}g_{1}(5n)q^{n}
&=\dfrac{\varphi(q)}{(q;q)_{\infty}^{2}(q^{2};q^{2})_{\infty}}\left(f(q^{18},q^{22})^{2}-q^{8}f(q^{2},q^{38})^{2}\right)\nonumber\\
 &\quad+\dfrac{2\psi(q^{2})}{(q;q)_{\infty}^{2}(q^{2};q^{2})_{\infty}}\big(q^{5}f(q^{12},q^{28})f(q^{2},q^{38})+q^{6}
 f(q^{8},q^{32})f(q^{2},q^{38})\nonumber\\
 &\quad-qf(q^{12},q^{28})f(q^{18},q^{22})-q^{2}f(q^{8},q^{32})f(q^{18},q^{22})\big).
\end{align}

On the other hand,
\begin{align}
\dfrac{1}{(q,q^{4};q^{5})_{\infty}^{2}} &=\dfrac{(q^{2},q^{3},q^{5};q^{5})_{\infty}^{2}}{(q;q)_{\infty}^{2}} =\dfrac{1}{(q;q)_{\infty}^{2}}\sum_{m,n=-\infty}^{\infty}(-1)^{m+n}q^{(5m^{2}+m)/2+(5n^{2}+n)/2}\nonumber\\
 &=\dfrac{1}{(q;q)_{\infty}^{2}}\Bigg(\sum_{r,s=-\infty}^{\infty}q^{(5(r+s)^{2}+(r+s))/2+(5(r-s)^{2}+(r-s))/2}\nonumber\\
&\quad\quad-\sum_{r,s=-\infty}^{\infty}q^{(5(r+s-1)^{2}+(r+s-1))/2+(5(r-s)^{2}+(r-s))/2}\Bigg)\nonumber\\
 &=\dfrac{1}{(q;q)_{\infty}^{2}}\left(\sum_{m,n=-\infty}^{\infty}q^{5m^{2}+m+5n^{2}}-\sum_{m,n=-\infty}^{\infty}q^{5m^{2}+5m+5n^{2}+4n}\right)\nonumber\\
 &=\dfrac{\varphi(q^{5})}{(q;q)_{\infty}^{2}}\sum_{m=-\infty}^{\infty}q^{5m^{2}+m}-\dfrac{2q^{2}\psi(q^{10})}{(q;q)_{\infty}^{2}}\sum_{m=-\infty}^{\infty}q^{5m^{2}+4m}
 \nonumber\\
 &=\dfrac{\varphi(q^{5})}{(q;q)_{\infty}^{2}}\left(\sum_{m=-\infty}^{\infty}q^{20m^{2}+2m}+q^{4}\sum_{m=-\infty}^{\infty}q^{20m^{2}+18m}\right)\nonumber\\
 &\quad-\dfrac{2\psi(q^{10})}{(q;q)_{\infty}^{2}}\left(q^{2}\sum_{m=-\infty}^{\infty}q^{20m^{2}+8m}+q^{3}\sum_{m=-\infty}^{\infty}q^{20m^{2}+12m}\right).\label{iden-1}
\end{align}

Moreover,
\begin{align}
\dfrac{1}{(q^{2},q^{8};q^{10})_{\infty}} &=\dfrac{(q^{4},q^{6},q^{10};q^{10})_{\infty}}{(q^{2};q^{2})_{\infty}}=\dfrac{1}{(q^{2};q^{2})_{\infty}}
\sum_{m=-\infty}^{\infty}(-1)^{m}q^{5m^{2}+m}\nonumber\\
 &=\dfrac{1}{(q^{2};q^{2})_{\infty}}\left(\sum_{m=-\infty}^{\infty}q^{20m^{2}+2m}-q^{4}\sum_{m=-\infty}^{\infty}q^{20m^{2}+18m}\right).\label{iden-2}
\end{align}

Combining \eqref{iden-1} and \eqref{iden-2} yields
\begin{align}
 &\dfrac{1}{(q,q^{4};q^{5})_{\infty}^{2}(q^{2},q^{8};q^{10})_{\infty}}\nonumber\\
 =&\Bigg(\dfrac{\varphi(q^{5})}{(q;q)_{\infty}^{2}}\left(\sum_{m=-\infty}^{\infty}q^{20m^{2}+2m}+q^{4}\sum_{m=-\infty}^{\infty}q^{20m^{2}+18m}\right)\nonumber\\
 &\quad-\dfrac{2\psi(q^{10})}{(q;q)_{\infty}^{2}}\left(q^{2}\sum_{m=-\infty}^{\infty}q^{20m^{2}+8m}+q^{3}\sum_{m=-\infty}^{\infty}q^{20m^{2}+12m}\right)\Bigg)\nonumber\\
 &\quad\times\dfrac{1}{(q^{2};q^{2})_{\infty}}\left(\sum_{m=-\infty}^{\infty}q^{20m^{2}+2m}-q^{4}\sum_{m=-\infty}^{\infty}q^{20m^{2}+18m}\right)\nonumber\\
 =&\dfrac{\varphi(q^{5})}{(q;q)_{\infty}^{2}(q^{2};q^{2})_{\infty}}\left(f(q^{18},q^{22})^{2}-q^{8}f(q^{2},q^{38})^{2}\right)\nonumber\\
 &\quad+\dfrac{2\psi(q^{10})}{(q;q)_{\infty}^{2}(q^{2};q^{2})_{\infty}}\big(q^{6}f(q^{12},q^{28})f(q^{2},q^{38})+q^{7}f(q^{8},q^{32})f(q^{2},q^{38})\nonumber\\
 &\quad-q^{2}f(q^{12},q^{28})f(q^{18},q^{22})-q^{3}f(q^{8},q^{32})f(q^{18},q^{22})\big).\label{iden:A0}
\end{align}

Eq. \eqref{a1:5n} follows from \eqref{iden:M-N}, \eqref{gf:a5n}, and \eqref{iden:A0}.

Next we are ready to prove \eqref{a1:5n+1}.

Following the same line of proving \eqref{gf:a5n}, we obtain
\begin{align}\label{gf:a5n+1}
\sum_{n=0}^{\infty}g_{1}(5n+1)q^{n}
&=\dfrac{2\varphi(q)}{(q;q)_{\infty}^{2}(q^{2};q^{2})_{\infty}}\left(q^{5}f(q^{10},q^{30})f(q^{2},q^{38})-qf(q^{10},q^{30})f(q^{18},q^{22})\right)\nonumber\\
 &\quad+\dfrac{2\psi(q^{2})}{(q;q)_{\infty}^{2}(q^{2};q^{2})_{\infty}}\big(f(q^{20},q^{20})f(q^{18},q^{22})+2q^{5}
 f(q^{40},q^{120})f(q^{18},q^{22})\nonumber\\
 &\quad-q^{4}f(q^{20},q^{20})f(q^{2},q^{38})-2q^{9}f(q^{40},q^{120})f(q^{2},q^{38})\big)\nonumber\\
 &:=\dfrac{2\varphi(q)}{(q;q)_{\infty}^{2}(q^{2};q^{2})_{\infty}}M_{2}(q)+\dfrac{2\psi(q^{2})}{(q;q)_{\infty}^{2}(q^{2};q^{2})_{\infty}}N_{2}(q),
\end{align}
say.

Taking $(a,b,c,d)=(-q^{4},-q^{16},-q^{6},-q^{14})$ in \eqref{ff},
\begin{align*}
f(q^{10},q^{30})f(q^{18},q^{22})-q^{4}f(q^{10},q^{30})f(q^{2},q^{38})=f(-q^{4},-q^{16})f(-q^{6},-q^{14}),
\end{align*}
which yields
\begin{align}
M_{2}(q) &=-qf(-q^{4},-q^{16})f(-q^{6},-q^{14}).\label{represe:S-final}
\end{align}

In view of \eqref{iden:psi} and \eqref{def:psi},
\begin{align}
 N_{2}(q)=&f(q^{20},q^{20})f(q^{18},q^{22})-2q^{9}f(q^{40},q^{120})f(q^{2},q^{38})\notag\\
 &\quad+2q^{5}f(q^{40},q^{120})f(q^{18},q^{22})-q^{4}f(q^{20},q^{20})f(q^{2},q^{38})\notag\\
 =&f(q^{20},q^{20})f(q^{18},q^{22})-q^{9}f(1,q^{40})f(q^{2},q^{38})\notag\\
 &\quad+q^{5}f(1,q^{40})f(q^{18},q^{22})-q^{4}f(q^{20},q^{20})f(q^{2},q^{38}).\label{relation}
\end{align}
Similarly, putting $(a,b,c,d)=(-q^{9},-q^{11},-q^{11},-q^{9})$ in \eqref{ff},
\begin{align}
f(q^{20},q^{20})f(q^{18},q^{22})-q^{9}f(1,q^{40})f(q^{2},q^{38})=f(-q^{9},-q^{11})^{2}.\label{iden:T1}
\end{align}
Picking $(a,b,c,d)=(-q,-q^{19},-q^{19},-q)$ in \eqref{ff},
\begin{align}
f(q^{20},q^{20})f(q^{2},q^{38})-qf(1,q^{40})f(q^{18},q^{22})=f(-q,-q^{19})^{2}.\label{iden:T2}
\end{align}
Taking $(a,b,c,d)=(-q^{4},-q^{6},q^{5},q^{5})$ in \eqref{ff},
\begin{align}
f(-q^{9},-q^{11})^{2}-q^{4}f(-q,-q^{19})^{2}=f(q^{5},q^{5})f(-q^{4},-q^{6})=\varphi(q^{5})f(-q^{4},-q^{6}).\label{iden:T3}
\end{align}
With the help of \eqref{relation}--\eqref{iden:T3},
\begin{align}
N_{2}(q) &=\varphi(q^{5})f(-q^{4},-q^{6}).\label{represe:T-final}
\end{align}

Substituting \eqref{represe:S-final} and \eqref{represe:T-final} into \eqref{gf:a5n+1},
\begin{align}\label{simp-gf:a5n+1}
\sum_{n=0}^{\infty}g_{1}(5n+1)q^{n} &=\dfrac{2\psi(q^{2})\varphi(q^{5})f(-q^{4},-q^{6})}{(q;q)_{\infty}^{2}(q^{2};q^{2})_{\infty}}
-\dfrac{2q\varphi(q)f(-q^{4},-q^{16})f(-q^{6},-q^{14})}{(q;q)_{\infty}^{2}(q^{2};q^{2})_{\infty}}\nonumber\\
 &=\dfrac{2\psi(q^{2})\varphi(q^{5})f(-q^{4},-q^{6})}{(q;q)_{\infty}^{2}(q^{2};q^{2})_{\infty}}-\dfrac{2q\varphi(q)\psi(q^{10})f(-q^{4},-q^{6})}
 {(q;q)_{\infty}^{2}(q^{2};q^{2})_{\infty}}.
\end{align}

On the other hand,
\begin{align}
 \dfrac{2}{(q,q^{2},q^{3},q^{4};q^{5})_{\infty}(q^{2},q^{8};q^{10})_{\infty}} &=\dfrac{2f(-q^{2},-q^{3})f(-q,-q^{4})f(-q^{4},-q^{6})}{(q;q)_{\infty}^{2}(q^{2};q^{2})_{\infty}}\nonumber\\
 &=\dfrac{2(q;q)_{\infty}(q^{5};q^{5})_{\infty}f(-q^{4},-q^{6})}{(q;q)_{\infty}^{2}(q^{2};q^{2})_{\infty}}.\label{simp:a1-5n+1}
\end{align}
In light of \eqref{phi-psi}, \eqref{simp-gf:a5n+1}, and \eqref{simp:a1-5n+1}, we obtain \eqref{a1:5n+1}.

The proofs of \eqref{a1:5n+2} and \eqref{a1:5n+4} are similar to that of \eqref{a1:5n}.

The proof of Theorem \ref{beau THM-2} is similar to Theorem \ref{beau THM-1}.

\section{Proofs of Theorems \ref{THM:relation} and \ref{THM:relation-2}}\label{sec:THM-relation}
We only prove \eqref{a2b2-relation1}, and the rest can be proved similarly.

From \cite{Tang2018}, we have the following representation for $\sum\limits_{n=0}^{\infty}g_{1,2,5}(n)q^{n}$:
\begin{align}
 \sum_{n=0}^{\infty} &g_{1,2,5}(n)q^{n}=\dfrac{(q^{10};q^{10})_{\infty}^{4}(q^{80};q^{80})_{\infty}^{5}f(q,q^{4})}{(q^{5};q^{5})_{\infty}^{5}(q^{20};q^{20})_{\infty}^{2}
 (q^{40};q^{40})_{\infty}^{2}(q^{160};q^{160})_{\infty}^{2}}\nonumber\\
 &\quad\times\Bigg(\sum_{n=-\infty}^{\infty}q^{40n^{2}+12n}
 -q^{4}\sum_{n=-\infty}^{\infty}q^{40n^{2}+28n}\Bigg)\nonumber\\
 &\quad+\dfrac{2(q^{10};q^{10})_{\infty}^{4}(q^{160};q^{160})_{\infty}^{2}f(q,q^{4})}
 {(q^{5};q^{5})_{\infty}^{5}(q^{20};q^{20})_{\infty}^{2}(q^{80};q^{80})_{\infty}}\Bigg(q^{14}\sum_{n=-\infty}^{\infty}q^{40n^{2}+28n}-q^{10}
 \sum_{n=-\infty}^{\infty}q^{40n^{2}+12n}\Bigg)\notag\\
 &\quad+\dfrac{2(q^{20};q^{20})_{\infty}^{2}f(q^{30},q^{50})f(q,q^{4})}{(q^{5};q^{5})_{\infty}^{3}(q^{10};q^{10})_{\infty}^{2}}
 \Bigg(q\sum_{n=-\infty}^{\infty}q^{40n^{2}+2n}-q^{10}\sum_{n=-\infty}^{\infty}q^{40n^{2}+38n}\notag\\
 &\quad+q^{4}\sum_{n=-\infty}^{\infty}q^{40n^{2}+22n}-q^{3}\sum_{n=-\infty}^{\infty}q^{40n^{2}+18n}\Bigg)\notag\\
 &\quad+\dfrac{2(q^{20};q^{20})_{\infty}^{2}f(q^{10},q^{70})f(q,q^{4})}{(q^{5};q^{5})_{\infty}^{3}(q^{10};q^{10})_{\infty}^{2}}
 \Bigg(q^{15}\sum_{n=-\infty}^{\infty}q^{40n^{2}+38n}-q^{9}\sum_{n=-\infty}^{\infty}q^{40n^{2}+22n}\notag\\
 &\quad+q^{8}\sum_{n=-\infty}^{\infty}q^{40n^{2}+18n}-q^{6}\sum_{n=-\infty}^{\infty}q^{40n^{2}+2n}\Bigg).
\end{align}

With the aid of \eqref{iden-1} and \eqref{iden-2},
\begin{align}
 &(-q^{2},-q^{3};q^{5})_{\infty}^{2}=\dfrac{(-q^{2},-q^{3},q^{5};q^{5})_{\infty}^{2}}{(q^{5};q^{5})_{\infty}^{2}}\nonumber\\
 =&\left(\dfrac{(q^{10};q^{10})_{\infty}^{5}}{(q^{5};q^{5})_{\infty}^{4}(q^{20};q^{20})_{\infty}^2}
 \sum_{m=-\infty}^{\infty}q^{20m^{2}+2m}+\dfrac{q^{4}(q^{10};q^{10})_{\infty}^{5}}{(q^{5};q^{5})_{\infty}^{4}(q^{20};q^{20})_{\infty}^2}
\sum_{m=-\infty}^{\infty}q^{20m^{2}+18m}\right)\nonumber\\
 &\quad+\left(\dfrac{2q^{2}(q^{20};q^{20})_{\infty}^{2}}{(q^{5};q^{5})_{\infty}^{2}(q^{10};q^{10})_{\infty}}\sum_{m=-\infty}^{\infty}q^{20m^{2}+8m}
 +\dfrac{2q^{3}(q^{20};q^{20})_{\infty}^{2}}{(q^{5};q^{5})_{\infty}^{2}(q^{10};q^{10})_{\infty}}\sum_{m=-\infty}^{\infty}q^{20m^{2}+12m}\right)\label{iden-3}
\end{align}
and
\begin{align}
(q^{4},q^{6};q^{10})_{\infty}=\dfrac{1}{(q^{10};q^{10})_{\infty}}\left(\sum_{m=-\infty}^{\infty}q^{20m^{2}+2m}-q^{4}\sum_{m=-\infty}^{\infty}q^{20m^{2}+18m}\right).
\label{iden-4}
\end{align}
Combining \eqref{iden-3} and \eqref{iden-4} as well as following the similar strategy of proving \eqref{iden-1}, we obtain
\begin{align*}
\sum_{n=0}^{\infty} &g_{2,4,5}(n)q^{n}
=\dfrac{f(q^{2},q^{3})}{(q^{5};q^{5})_{\infty}(q^{10};q^{10})_{\infty}}\Bigg(\sum_{m=-\infty}^{\infty}q^{20m^{2}+2m}-q^{4}\sum_{m=-\infty}^{\infty}q^{20m^{2}+18m}\Bigg)\\
 &\quad\times\Bigg(\dfrac{(q^{10};q^{10})_{\infty}^{5}}{(q^{5};q^{5})_{\infty}^{4}(q^{20};q^{20})_{\infty}^{2}}
 \sum_{n=-\infty}^{\infty}q^{20n^{2}+2n}+\dfrac{q^{4}(q^{10};q^{10})_{\infty}^{5}}{(q^{5};q^{5})_{\infty}^{4}(q^{20};q^{20})_{\infty}^{2}}
 \sum_{n=-\infty}^{\infty}q^{20n^{2}+18n}\\
 &\quad+\dfrac{2q^{2}(q^{20};q^{20})_{\infty}^{2}}{(q^{5};q^{5})_{\infty}^{2}(q^{10};q^{10})_{\infty}}\sum_{n=-\infty}^{\infty}q^{20n^{2}+8n}
 +\dfrac{2q^{3}(q^{20};q^{20})_{\infty}^{2}}{(q^{5};q^{5})_{\infty}^{2}(q^{10};q^{10})_{\infty}}\sum_{n=-\infty}^{\infty}q^{20n^{2}+12n}\Bigg)\\
 &=\dfrac{(q^{10};q^{10})_{\infty}^{4}(q^{80};q^{80})_{\infty}^{5}f(q^{2},q^{3})}{(q^{5};q^{5})_{\infty}^{5}(q^{20};q^{20})_{\infty}^{2}(q^{40};q^{40})_{\infty}^{2}
 (q^{160};q^{160})_{\infty}^{2}}
 \Bigg(\sum_{n=-\infty}^{\infty}q^{40n^{2}+4n}-q^{8}\sum_{n=-\infty}^{\infty}q^{40n^{2}+36n}\Bigg)\\
 &\quad+\dfrac{2(q^{10};q^{10})_{\infty}^{4}(q^{160};q^{160})_{\infty}^{2}f(q^{2},q^{3})}{(q^{5};q^{5})_{\infty}^{5}(q^{20};q^{20})_{\infty}^{2}(q^{80};q^{80})_{\infty}}
 \Bigg(q^{18}\sum_{n=-\infty}^{\infty}q^{40n^{2}+36n}-q^{10}\sum_{n=-\infty}^{\infty}q^{40n^{2}+4n}\Bigg)\\
 &\quad+\dfrac{2(q^{20};q^{20})_{\infty}^{2}f(q^{30},q^{50})f(q^{2},q^{3})}{(q^{5};q^{5})_{\infty}^{3}(q^{10};q^{10})_{\infty}^{2}}
 \Bigg(q^{2}\sum_{n=-\infty}^{\infty}q^{40n^{2}+6n}-q^{9}\sum_{n=-\infty}^{\infty}q^{40n^{2}+34n}\\
 &\quad+q^{3}\sum_{n=-\infty}^{\infty}q^{40n^{2}+14n}-q^{6}\sum_{n=-\infty}^{\infty}q^{40n^{2}+26n}\Bigg)\\
 &\quad+\dfrac{2(q^{20};q^{20})_{\infty}^{2}f(q^{10},q^{70})f(q^{2},q^{3})}{(q^{5};q^{5})_{\infty}^{3}(q^{10};q^{10})_{\infty}^{2}}
 \Bigg(q^{14}\sum_{n=-\infty}^{\infty}q^{40n^{2}+34n}-q^{8}\sum_{n=-\infty}^{\infty}q^{40n^{2}+14n}\\
 &\quad+q^{11}\sum_{n=-\infty}^{\infty}q^{40n^{2}+26n}-q^{7}\sum_{n=-\infty}^{\infty}q^{40n^{2}+6n}\Bigg).
\end{align*}

Define
\begin{align*}
S_{1} &:=f(q,q^{4})\sum_{n=-\infty}^{\infty}q^{40n^{2}+12n}-q^{4}f(q,q^{4})\sum_{n=-\infty}^{\infty}q^{40n^{2}+28n},\\
S_{2} &:=q^{14}f(q,q^{4})\sum_{n=-\infty}^{\infty}q^{40n^{2}+28n}-q^{10}f(q,q^{4})\sum_{n=-\infty}^{\infty}q^{40n^{2}+12n},\\
S_{3} &:=qf(q,q^{4})\sum_{n=-\infty}^{\infty}q^{40n^{2}+2n}-q^{10}f(q,q^{4})\sum_{n=-\infty}^{\infty}q^{40n^{2}+38n},\\
S_{4} &:=q^{4}f(q,q^{4})\sum_{n=-\infty}^{\infty}q^{40n^{2}+22n}-q^{3}f(q,q^{4})\sum_{n=-\infty}^{\infty}q^{40n^{2}+18n},\\
S_{5} &:=q^{15}f(q,q^{4})\sum_{n=-\infty}^{\infty}q^{40n^{2}+38n}-q^{9}f(q,q^{4})\sum_{n=-\infty}^{\infty}q^{40n^{2}+22n},\\
S_{6} &:=q^{8}f(q,q^{4})\sum_{n=-\infty}^{\infty}q^{40n^{2}+18n}-q^{6}f(q,q^{4})\sum_{n=-\infty}^{\infty}q^{40n^{2}+2n},\\
T_{1} &:=f(q^{2},q^{3})\sum_{n=-\infty}^{\infty}q^{40n^{2}+4n}-q^{8}f(q^{2},q^{3})\sum_{n=-\infty}^{\infty}q^{40n^{2}+36n},\\
T_{2} &:=q^{18}f(q^{2},q^{3})\sum_{n=-\infty}^{\infty}q^{40n^{2}+36n}-q^{10}f(q^{2},q^{3})\sum_{n=-\infty}^{\infty}q^{40n^{2}+4n},\\
T_{3} &:=q^{2}f(q^{2},q^{3})\sum_{n=-\infty}^{\infty}q^{40n^{2}+6n}-q^{9}f(q^{2},q^{3})\sum_{n=-\infty}^{\infty}q^{40n^{2}+34n},\\
T_{4} &:=q^{3}f(q^{2},q^{3})\sum_{n=-\infty}^{\infty}q^{40n^{2}+14n}-q^{6}f(q^{2},q^{3})\sum_{n=-\infty}^{\infty}q^{40n^{2}+26n},\\
T_{5} &:=q^{14}f(q^{2},q^{3})\sum_{n=-\infty}^{\infty}q^{40n^{2}+34n}-q^{8}f(q^{2},q^{3})\sum_{n=-\infty}^{\infty}q^{40n^{2}+14n},\\
T_{6} &:=q^{11}f(q^{2},q^{3})\sum_{n=-\infty}^{\infty}q^{40n^{2}+26n}-q^{7}f(q^{2},q^{3})\sum_{n=-\infty}^{\infty}q^{40n^{2}+6n}.
\end{align*}

Next, we prove that
\begin{align*}
H_{5,1}(S_{i})=H_{5,2}(T_{i})\quad \textrm{for}~1\leq i\leq6.
\end{align*}
We only prove the case $H_{5,1}(S_{1})=H_{5,2}(T_{1})$ here because the proofs of remaining cases are similar.

Notice that
\begin{align*}
f(q^{2},q^{3}) &=\sum_{m=-\infty}^{\infty}q^{(5m^{2}+m)/2}\\
 &=\sum_{m=-\infty}^{\infty}q^{10m^{2}+m}+q^{2}\sum_{m=-\infty}^{\infty}q^{10m^{2}+9m}\\
 &=\sum_{m=-\infty}^{\infty}q^{40m^{2}+2m}+q^{9}\sum_{m=-\infty}^{\infty}q^{40m^{2}+38m}+q^{2}\sum_{m=-\infty}^{\infty}q^{40m^{2}+18m}
+q^{3}\sum_{m=-\infty}^{\infty}q^{40m^{2}+22m}
\end{align*}
and
\begin{align*}
f(q,q^{4})&=\sum_{m=-\infty}^{\infty}q^{40m^{2}+6m}+q^{7}\sum_{m=-\infty}^{\infty}q^{40m^{2}+34m}
 +q\sum_{m=-\infty}^{\infty}q^{40m^{2}+14m}+q^{4}\sum_{m=-\infty}^{\infty}q^{40m^{2}+26m}.
\end{align*}

Therefore,
\begin{align*}
S_{1} &=P_{1}+P_{2}+P_{3}+P_{4}-P_{5}-P_{6}-P_{7}-P_{8},\\
T_{1} &=Q_{1}+Q_{2}+Q_{3}+Q_{4}-Q_{5}-Q_{6}-Q_{7}-Q_{8},
\end{align*}
where
\begin{align*}
P_{1} &=\sum_{m,n=-\infty}^{\infty}q^{40m^{2}+6m+40n^{2}+12n},~P_{2}=q^{7}\sum_{m,n=-\infty}^{\infty}q^{40m^{2}+34m+40n^{2}+12n},\\
P_{3} &=q\sum_{m,n=-\infty}^{\infty}q^{40m^{2}+14m+40n^{2}+12n},~P_{4}=q^{4}\sum_{m,n=-\infty}^{\infty}q^{40m^{2}+26m+40n^{2}+12n},\\
P_{5} &=q^{4}\sum_{m,n=-\infty}^{\infty}q^{40m^{2}+6m+40n^{2}+28n},~P_{6}=q^{11}\sum_{m,n=-\infty}^{\infty}q^{40m^{2}+34m+40n^{2}+28n},\\
P_{7} &=q^{5}\sum_{m,n=-\infty}^{\infty}q^{40m^{14}+14m+40n^{2}+28n},~P_{8}=q^{8}\sum_{m,n=-\infty}^{\infty}q^{40m^{2}+26m+40n^{2}+28n},\\
Q_{1} &=q^{9}\sum_{m,n=-\infty}^{\infty}q^{40m^{2}+38m+40n^{2}+4n},~Q_{2}=\sum_{m,n=-\infty}^{\infty}q^{40m^{2}+2m+40n^{2}+4n},\\
Q_{3} &=q^{2}\sum_{m,n=-\infty}^{\infty}q^{40m^{2}+18m+40n^{2}+4n},~Q_{4}=q^{3}\sum_{m,n=-\infty}^{\infty}q^{40m^{2}+22m+40n^{2}+4n},\\
Q_{5} &=q^{17}\sum_{m,n=-\infty}^{\infty}q^{40m^{2}+38m+40n^{2}+36n},~Q_{6}=q^{8}\sum_{m,n=-\infty}^{\infty}q^{40m^{2}+2m+40n^{2}+36n},\\
Q_{7} &=q^{10}\sum_{m,n=-\infty}^{\infty}q^{40m^{2}+18m+40n^{2}+36n},~Q_{8}=q^{11}\sum_{m,n=-\infty}^{\infty}q^{40m^{2}+22m+40n^{2}+36n}.
\end{align*}

Following the similar strategy of proving \eqref{S1:5}, we deduce that
\begin{align*}
H_{5,1}(P_{i}) &=H_{5,2}(Q_{i})\quad \textrm{for}~1\leq i\leq8.
\end{align*}
This establishes \eqref{a2b2-relation1}.

Finally, we are ready to prove \eqref{q-iden-1}--\eqref{q-iden-4}.

It follows easily from \eqref{JTP-identity} that
\begin{align}
f(-q^{2s},-q^{2t}) &=\dfrac{(q^{2s+2t};q^{2s+2t})_{\infty}}{(q^{s+t};q^{s+t})_{\infty}^{2}}f(q^{s},q^{t})f(-q^{s},-q^{t}),\quad s,t\in\mathbb{N}_{+},\label{useful-iden-1}\\
f(q,q^{4})f(q^{2},q^{3}) &=\dfrac{(q^{2};q^{2})_{\infty}(q^{5};q^{5})_{\infty}^{3}}{(q;q)_{\infty}(q^{10};q^{10})_{\infty}},\label{useful-iden-2}\\
f(-q^{2},-q^{3})f(-q^{4},-q^{6}) &=(q^{5};q^{5})_{\infty}(q^{2},q^{3},q^{4},q^{6},q^{7},q^{8},q^{10};q^{10})_{\infty}\nonumber\\
 &=\dfrac{(q^{2};q^{2})_{\infty}(q^{5};q^{5})_{\infty}}{(q^{10};q^{10})_{\infty}}f(-q^{3},-q^{7}),\label{iden-ff-1}\\
f(-q,-q^{4})f(-q^{2},-q^{8}) &=(q^{5};q^{5})_{\infty}(q,q^{2},q^{4},q^{6},q^{8},q^{9},q^{10};q^{10})_{\infty}\nonumber\\
 &=\dfrac{(q^{2};q^{2})_{\infty}(q^{5};q^{5})_{\infty}}{(q^{10};q^{10})_{\infty}}f(-q,-q^{9}).\label{iden-ff-2}
\end{align}

On one hand, according to \eqref{useful-iden-1} and \eqref{useful-iden-2}, we obtain
\begin{align}
&(-q,-q^{4};q^{5})_{\infty}^{2}(q^{4},q^{6};q^{10})_{\infty}+(-q^{2},-q^{3};q^{5})_{\infty}^{2}(q^{2},q^{8};q^{10})_{\infty}\notag\\
 =&\dfrac{1}{(q^{5};q^{5})_{\infty}^{2}(q^{10};q^{10})_{\infty}}\bigg(f(q,q^{4})^{2}f(-q^{4},-q^{6})+f(q^{2},q^{3})^{2}f(-q^{2},-q^{8})\bigg)\notag\\
 =&\dfrac{1}{(q^{5};q^{5})_{\infty}^{4}}f(q,q^{4})f(q^{2},q^{3})\bigg(f(q,q^{4})f(-q^{2},-q^{3})+f(-q,-q^{4})f(q^{2},q^{3})\bigg)\notag\\
 =&\dfrac{(q^{2};q^{2})_{\infty}}{(q;q)_{\infty}(q^{5};q^{5})_{\infty}(q^{10};q^{10})_{\infty}}\bigg(f(q,q^{4})f(-q^{2},-q^{3})+f(-q,-q^{4})f(q^{2},q^{3})\bigg).
 \label{iden-P1-1}
\end{align}
Taking $(a,b,c,d)=(-q,-q^{4},q^{2},q^{3})$ in \eqref{ff},
\begin{align}
f(-q,-q^{4})f(q^{2},q^{3})=f(-q^{3},-q^{7})f(-q^{4},-q^{6})-qf(-q,-q^{9})f(-q^{2},-q^{8}).\label{ff-iden-1}
\end{align}
Picking $(a,b,c,d)=(q,q^{4},-q^{2},-q^{3})$ in \eqref{ff},
\begin{align}
f(q,q^{4})f(-q^{2},-q^{3})=f(-q^{3},-q^{7})f(-q^{4},-q^{6})+qf(-q,-q^{9})f(-q^{2},-q^{8}).\label{ff-iden-2}
\end{align}
Substituting \eqref{ff-iden-1} and \eqref{ff-iden-2} into \eqref{iden-P1-1},
\begin{align*}
(-q,-q^{4};q^{5})_{\infty}^{2}(q^{4},q^{6};q^{10})_{\infty} &+(-q^{2},-q^{3};q^{5})_{\infty}^{2}(q^{2},q^{8};q^{10})_{\infty}\\
 &\quad=\dfrac{2(q^{2};q^{2})_{\infty}}{(q;q)_{\infty}(q^{5};q^{5})_{\infty}(q^{10};q^{10})_{\infty}}f(-q^{3},-q^{7})f(-q^{4},-q^{6}).
\end{align*}

On the other hand, with the help of \eqref{useful-iden-2} and \eqref{iden-ff-1},
\begin{align}
 &\dfrac{2(q^{10};q^{10})_{\infty}^{3}}{(q^{2};q^{2})_{\infty}(q^{5};q^{5})_{\infty}^{2}}(-q,-q^{4};q^{5})_{\infty}(q^{4},q^{6};q^{10})_{\infty}^{3}\notag\\
 =&\dfrac{2}{(q^{2};q^{2})_{\infty}(q^{5};q^{5})_{\infty}^{3}}f(q,q^{4})f(-q^{4},-q^{6})^{3}\notag\\
 =&\dfrac{2(q^{10};q^{10})_{\infty}}{(q^{2};q^{2})_{\infty}(q^{5};q^{5})_{\infty}^{5}}\bigg(\left(f(q,q^{4})f(q^{2},q^{3})\right)\left(f(-q^{2},-q^{3})f(-q^{4},-q^{6})\right)f(-q^{4},-q^{6})\bigg)\notag\\
 =&\dfrac{2(q^{2};q^{2})_{\infty}}{(q;q)_{\infty}(q^{5};q^{5})_{\infty}(q^{10};q^{10})_{\infty}}f(-q^{3},-q^{7})f(-q^{4},-q^{6}).\notag
\end{align}
This establishes \eqref{q-iden-1}.

According to \eqref{useful-iden-1}--\eqref{iden-ff-2} and \eqref{ff-iden-1},
\begin{align*}
 &(-q,-q^{4};q^{5})_{\infty}(q^{4},q^{6};q^{10})_{\infty}^{3}-q(-q^{2},-q^{3};q^{5})_{\infty}(q^{2},q^{8};q^{10})_{\infty}^{3}\\ =&\dfrac{1}{(q^{5};q^{5})_{\infty}(q^{10};q^{10})_{\infty}^{3}}\bigg(f(q,q^{4})f(-q^{4},-q^{6})^{3}-qf(q^{2},q^{3})f(-q^{2},-q^{8})^{3}\bigg)\\
 =&\dfrac{1}{(q^{5};q^{5})_{\infty}^{3}(q^{10};q^{10})_{\infty}^{2}}f(q,q^{4})f(q^{2},q^{3})\bigg(\left(f(-q^{2},-q^{3})f(-q^{4},-q^{6})\right)f(-q^{4},-q^{6})\\
 &\quad-q\left(f(-q,-q^{4})f(-q^{2},-q^{8})\right)f(-q^{2},-q^{8})\bigg)\\
 =&\dfrac{(q^{2};q^{2})_{\infty}^{2}(q^{5};q^{5})_{\infty}}{(q;q)_{\infty}(q^{10};q^{10})_{\infty}^{4}}f(-q,-q^{4})f(q^{2},q^{3}).
\end{align*}

Also, using \eqref{useful-iden-1},
\begin{align*}
 &\dfrac{(q^{2};q^{2})_{\infty}(q^{5};q^{5})_{\infty}^{2}}{(q^{10};q^{10})_{\infty}^{3}}(-q^{2},-q^{3};q^{5})_{\infty}^{2}(q^{2},q^{8};q^{10})_{\infty}\\
 =&\dfrac{(q^{2};q^{2})_{\infty}}{(q^{10};q^{10})_{\infty}^{4}}f(q^{2},q^{3})^{2}f(-q^{2},-q^{8})\\
 =&\dfrac{(q^{2};q^{2})_{\infty}}{(q^{5};q^{5})_{\infty}^{2}(q^{10};q^{10})_{\infty}^{3}}\left(f(q,q^{4})f(q^{2},q^{3})\right)f(-q,-q^{4})f(q^{2},q^{3})\\
 =&\dfrac{(q^{2};q^{2})_{\infty}^{2}(q^{5};q^{5})_{\infty}}{(q;q)_{\infty}(q^{10};q^{10})_{\infty}^{4}}f(-q,-q^{4})f(q^{2},q^{3}).
\end{align*}
This proves \eqref{q-iden-3}.

The proofs of \eqref{q-iden-2} and \eqref{q-iden-4} are similar to those of \eqref{q-iden-1} and \eqref{q-iden-3}, respectively.

\section{Final remarks}
We close this paper with some remarks.
\begin{enumerate}[1)]
\item Following the same line of proving \eqref{a2b2-relation1}--\eqref{a3b3-relation2}, we can also prove
\begin{align}
a_{1,1,7}(7n+1) &=a_{3,3,7}(7n+3),\label{analog-relat-3}\\
a_{1,6,7}(7n+6) &=-a_{2,2,7}(7n+6),\\
a_{4,6,11}(11n+5) &=-a_{5,2,11}(11n+4),\\
a_{4,6,11}(11n+7) &=a_{5,2,11}(11n+6).\label{analog-relat-2}
\end{align}

There are other identities similar to \eqref{analog-relat-3}--\eqref{analog-relat-2} for $t=11$. Therefore it is natural to ask whether or not there exist some identities between $a_{r,s,t}(n)$ and $b_{r,s,t}(n)$ for arbitrary prime $t$, which parallel to \eqref{analog-relat-3}--\eqref{analog-relat-2}.

\item Following the similar method of proving \eqref{iden-vanish} in \cite{Tang2018}, we can also obtain
\begin{align}
g_{2}(5n+3)=h_{2}(5n+1)=0,\label{iden-vanish-2}
\end{align}
which parallels to \eqref{iden-vanish}.

Eqs. \eqref{q-iden-2} and \eqref{iden-vanish-2} imply
\begin{align*}
h_{2}(5n+3) &\equiv0\pmod{2}.
\end{align*}

Furthermore, there are some results similar to \eqref{a2b2-relation1}--\eqref{final-relation1} in another types of $q$-series expansions. Relating to \eqref{h,r-s-t} and \eqref{q-iden-1}, define
\begin{align*}
(-q,-q^{4};q^{5})_{\infty}^{2}(q^{4},q^{6};q^{10})_{\infty}^{2}(q^{2},q^{8};q^{10})_{\infty} &=\sum_{n=0}^{\infty}\widehat{g_{1}}(n)q^{n},\\
(-q^{2},-q^{3};q^{5})_{\infty}^{2}(q^{2},q^{8};q^{10})_{\infty}^{2}(q^{4},q^{6};q^{10})_{\infty} &=\sum_{n=0}^{\infty}\widehat{h_{1}}(n)q^{n}.
\end{align*}
Following the similar strategy of proving \eqref{a2b2-relation1}, we can also obtain
\begin{align*}
\widehat{g_{1}}(5n) &=-\widehat{h_{1}}(5n),\\
\widehat{g_{1}}(5n+2) &=-\widehat{h_{1}}(5n+1)=0,\\
\widehat{g_{1}}(5n+3) &=-\widehat{h_{1}}(5n+3).
\end{align*}
Of course, we can also obtain similar results for \eqref{q-iden-2}--\eqref{q-iden-4}.

\item We also learn from Nayandeep Deka Baruah and Mandeep Kaur \cite{BK2018} that they have provided new proofs of \eqref{a3b3-relation1}--\eqref{analog-relati-2}. Their proofs rely highly on two known $q$-identities \cite[Eqs. (40.1.1) and (41.1.5)]{Hirb2017} involving Ramanujan's continued fractions.

\item Finally, with the help of computer, the signs of coefficients in $q$-series \eqref{gf:g2} and \eqref{gf:h2} appear to be periodic.
\begin{conjecture}
For any integer $n\geq0$,
\begin{align}
g_{2}(5n) &>0,\label{ineq-begin}\\
g_{2}(5n+1) &<0,\\
g_{2}(5n+2) &>0,\\
g_{2}(5n+4) &<0,\\
h_{2}(5n) &>0,\\
h_{2}(5n+2) &<0,\\
h_{2}(5n+3) &<0,\\
h_{3}(5n+4) &>0.\label{ineq-end}
\end{align}
\end{conjecture}
It would be interesting to find an elementary proof of \eqref{ineq-begin}--\eqref{ineq-end}.
\end{enumerate}

\section*{Acknowledgement}
The authors are indebted to Shishuo Fu for his helpful comments on a preliminary version of this paper. The first author was supported by the National Natural Science Foundation of China (No.~11501061) and the Fundamental Research Funds for the Central Universities (No.~2018CDXYST0024). The second author was supported by the National Natural Science Foundation of China (No.~11571143) and Jiangsu National Funds for Distinguished Young Scientists (No.~BK20180044).

\end{document}